\documentclass[reqno]{amsart}

\usepackage{amssymb,amsmath,amsthm}
\usepackage{enumitem}
\usepackage[frame,cmtip,arrow,matrix,line,graph,curve]{xy}
\usepackage[capitalise]{cleveref}

\newtheorem{theorem}{Theorem}[section]
\newtheorem*{thmnonum}{Theorem}
\newtheorem{proposition}[theorem]{Proposition}
\newtheorem{lemma}[theorem]{Lemma}
\newtheorem{corollary}[theorem]{Corollary}
\newtheorem{conjecture}[theorem]{Conjecture}

\theoremstyle{definition}

\DeclareMathOperator{\addOp}{add}
\DeclareMathOperator{\Coker}{Coker}
\DeclareMathOperator{\dimv}{\underline{dim}}
\DeclareMathOperator{\End}{End}
\DeclareMathOperator{\Ext}{Ext}
\DeclareMathOperator{\Hom}{Hom}
\DeclareMathOperator{\Img}{Im}
\DeclareMathOperator{\ind}{ind}
\DeclareMathOperator{\Ker}{Ker}
\DeclareMathOperator{\SL}{SL}
\DeclareMathOperator{\Tr}{Tr}

\newcommand{\C}{\mathcal{C}}
\newcommand{\D}{\mathcal{D}}
\newcommand{\lto}{\longrightarrow}
\renewcommand{\phi}{\varphi}

\newcommand{\matmapv}[2]{\left[\begin{array}{c} #1 \\ #2 \end{array}\right]}
\newcommand{\matmaph}[2]{\left[\begin{array}{cc} #1 & #2 \end{array}\right]}
\newcommand{\smatmapv}[2]{\left[\begin{smallmatrix} #1 \\ #2 \end{smallmatrix}\right]}
\newcommand{\smatmaph}[2]{\left[\begin{smallmatrix} #1 & #2  \end{smallmatrix}\right]}
\newcommand{\SES}[5]{0\lto #1\stackrel{#2}{\lto}#3\stackrel{#4}{\lto}#5\lto 0}
\newcommand{\set}[1]{\left\{#1\right\}}
\newcommand{\add}[1]{\addOp\!\left(#1\right)}
\newcommand{\HomC}{\Hom_{\C}}

\title{Christoffel Words and Markoff Triples: An Algebraic Approach}
\author{Alex Lasnier}
\address{D\'epartement de math\'ematiques, Universit\'e de Sherbrooke, 2500 boul. de l'Universit\'e, Sherbrooke (Qu\'ebec), J1K 2R1, Canada.}
\email{alex.lasnier@usherbrooke.ca}


\begin{document}

\begin{abstract}
We introduce a family of modules, called Markoff modules, generated by a cluster-mutation-like iterative process. We show that these modules are combinatorially similar to Christoffel words. Furthermore, we construct a bijective map between the set of Markoff module triples and the set of proper Markoff triples. This allows us to interpret the uniqueness conjecture for Markoff numbers within an algebraic framework.
\end{abstract}

\maketitle


\section{Introduction}

A Markoff triple is a triple of positive integers $a,b,c$ satisfying the Diophantine equation
\[ a^2 + b^2 + c^2 = 3abc \]
The numbers appearing in these triples, called Markoff numbers,
were first studied by Markoff in his work on the minima of indefinite binary quadratic forms \cite{M},\cite{M2}.
He showed that every such triple can be generated by certain simple arithmetic rules starting from $(1,1,1)$, thereby building a $3$-regular tree composed of all solutions to the equation. The uniqueness conjecture for Markoff numbers, first stated by Frobenius \cite{Fro}, claims that every Markoff number appears uniquely as the largest element of a Markoff triple (up to permutation). 

Christoffel words were introduced by their namesake in \cite{Ch} and more recently revitalised with an in depth treatment by Borel and Laubie \cite{BL}. These are words in a two letter alphabet constructed by encoding the discretisation of certain line segments in $\mathbb{R}^2$.
In \cite{R}, Reutenauer constructs a bijective map associating a Markoff triple to every Christoffel word. Following Cohn \cite{C1}, he accomplishes this using the Fricke identities \cite{Fri}; a strategy that will also be of central importance to us.

In this paper, we introduce a family of string modules, called Markoff modules, generated in triples by an iterative process inspired by the mutation of tilting objects in a cluster category \cite{BMRRT} and analogous to the tree construction of Markoff triples. The latter similarity is found to be more than superficial by explicitly defining a bijection between the set of Markoff module triples and the set of proper Markoff triples, commuting with the structure maps of the respective trees. This map is defined on a Markoff module by $M(w)\mapsto \tfrac{1}{3}\Tr\phi(w)$ where $\phi(w)$ is a matrix in $\SL_2(\mathbb{Z})$ built from the string $w$ that defines the module $M(w)$. 
Our main results can be summarised as follows:
\begin{thmnonum}
Let $\mathcal{T}$ be the set of Markoff module triples and $\mathcal{M}$ the set of proper Markoff triples. The map $\Phi:\mathcal{T}\to\mathcal{M}$ defined by \[ \Phi(M(w_1),M(w_2),M(w_3)) = \left( \tfrac{1}{3}\Tr\phi(w_1),\, \tfrac{1}{3}\Tr\phi(w_2),\, \tfrac{1}{3}\Tr\phi(w_3) \right) \]
is a binary tree isomorphism. 
Moreover, the uniqueness conjecture for Markoff numbers is equivalent to the injectivity of the map $M(w)\mapsto \tfrac{1}{3}\Tr\phi(w)$ where $M(w)$ is a proper Markoff module.
\end{thmnonum}

\section{Preliminaries}

\subsection{Binary Trees} 
Even though binary trees have been extensively studied in mathematics and computer science, we provide an alternative definition better suited to our algebraic framework. For a set $X$, we denote by $X^*$ the free monoid generated by $X$.

A \emph{(complete, infinite, rooted) binary tree} is a triple $(T,L,R)$ where $T$ is a countably infinite set and $L,R$ are two injective maps $T\to T$ satisfying:
\begin{enumerate}[label=(\arabic{*})]
  \item There is a unique $r_T\in T$ such that $r_T\notin \Img L \cup \Img R$. This element is called the \emph{root}.
  \item No element of $T$ belongs to both $\Img L$ and $\Img R$.
  \item For every $x\in T$, there is a $f_x\in\set{L,R}^*$ such that $f_x(r_T)=x$.
\end{enumerate}

\begin{lemma}\label{btunique}
Let $(T,L,R)$ be a binary tree.
For every $x\in T$, the element $f_x\in\set{L,R}^*$ such that $f_x(r_T)=x$ is unique.
\end{lemma}
\begin{proof}
Since $r_T\notin \Img L \cup \Img R$, $f=\mathbb{I}_T$ is the unique element of $\set{L,R}^*$ such that $f(r_T)=r_T$.
Now, suppose that $f_x\in\set{L,R}^*$ is the unique function satisfying $f_x(r_T)=x$ for some $x\in T$.
Let $y=L(x)$, then there is a $f_y\in\set{L,R}^*$ such that $f_y(r_T)=y$.
Since $y\in\Img L$, we have $y\notin \Img R$ thus $f_y$ can only be expressed as $f_y=L\circ g$ with $g\in\set{L,R}^*$.
Hence $L\circ g(r_T) = y = L(x)$ and the injectivity of $L$ implies that $g(r_T)=x$.
By the inductive hypothesis, we get $g = f_x$. Therefore $f_y = L\circ f_x$ is the unique element of $\set{L,R}^*$ such that $f_y(r_T)=y$.
Similarly, one can show that $f = R\circ f_x$ is the unique function such that $f(r_T) = R(x)$.
\end{proof}

A \emph{binary tree homomorphism} between two binary trees $\phi:(T_1,L_1,R_1)\to (T_2,L_2,R_2)$ is given by a map $\phi:T_1\to T_2$ such that $\phi L_1 = L_2 \phi$ and $\phi R_1 = R_2 \phi$.
\[ \xymatrix{
T_1 \ar[r]^{\phi} \ar[d]_{L_1} & T_2 \ar[d]^{L_2}  & & T_1 \ar[r]^{\phi} \ar[d]_{R_1} & T_2 \ar[d]^{R_2} \\
T_1 \ar[r]^{\phi}              & T_2               & & T_1 \ar[r]^{\phi}              & T_2 \\
} \]

Let $(T_1,L_1,R_1)$ and $(T_2,L_2,R_2)$ be binary trees.
We define a monoid isomorphism $i: \set{L_1,R_1}^* \to \set{L_2,R_2}^*$ by $i(L_1)=L_2$ and $i(R_1)=R_2$.
\begin{lemma}
Let $\phi:(T_1,L_1,R_1)\to (T_2,L_2,R_2)$ be a binary tree homomorphism and $f\in\set{L_1,R_1}^*$. 
We have $\phi\circ f = i(f)\circ \phi$.
\end{lemma}
\begin{proof}
The statement is clearly true when $f=\mathbb{I}_{T_1}$.
Assume that the equality holds for some $f\in\set{L_1,R_1}^*$.
We have
\[ \phi\circ (L_1\circ f) = L_2\circ (\phi\circ f) = L_2\circ i(f)\circ \phi = i(L_1)\circ i(f)\circ \phi = i(L_1\circ f)\circ \phi \]
and analogously $\phi\circ (R_1\circ f) = i(R_1\circ f)\circ \phi$.
\end{proof}

\begin{proposition}\label{btiso}
Let $\phi:(T_1,L_1,R_1)\to (T_2,L_2,R_2)$ be a binary tree homomorphism. Then $\phi$ is a bijection if and only if $\phi(\text{root } T_1) = \text{root } T_2$. In this case we say that $\phi$ is a \emph{binary tree isomorphism}.
\end{proposition}
\begin{proof}
Suppose that $\phi$ is a bijection. Then there exists $x\in T_1$ such that $\phi(x)=r_{T_2}$.
If $x$ was in $\Img L_1 \cup \Img R_1$ we would have, by the commutativity properties,
$r_{T_2} \in \Img L_2 \cup \Img R_2$; a contradiction. Thus $x\notin \Img L \cup \Img R$ which implies $x=r_{T_1}$.

Conversely, assume that $\phi(r_{T_1}) = r_{T_2}$. 
Let $y\in T_2$, then there is a $f_y\in\set{L_2,R_2}^*$ such that $f_y(r_{T_2}) = y$.
Let $x=f_x(r_{T_1})$ where $f_x$ is the unique element of $\set{L_1,R_1}^*$ such that $i(f_x)=f_y$.
We have
\[ \phi(x) = \phi(f_x(r_{T_1})) = i(f_x)(\phi(r_{T_1})) = f_y(r_{T_2}) = y \]
Hence $\phi$ is surjective.

Let $x,y\in T_1$ and suppose that $\phi(x)=\phi(y)$.
There exist $f_x,f_y\in\set{L_1,R_1}^*$ such that $f_x(r_{T_1}) = x$ and $f_y(r_{T_1}) = y$.
We get
\[ \phi(x) = \phi(f_x(r_{T_1})) = i(f_x)(\phi(r_{T_1})) = i(f_x)(r_{T_2})\]
and similarly $\phi(y) = i(f_y)(r_{T_2})$.
As a result $i(f_x)(r_{T_2}) = i(f_y)(r_{T_2})$ thus $i(f_x) = i(f_y)$ by \cref{btunique}.
Since $i$ is an isomorphism, we have $f_x = f_y$ and
\[ x = f_x(r_{T_1}) = f_y(r_{T_1}) = y \]
Therefore $\phi$ is injective.
\end{proof}

\subsection{Markoff Triples} 
A \emph{Markoff triple} is a solution $\set{a,b,c}$ in the positive integers of the equation
\[ x^2+y^2+z^2 = 3xyz \]
A Markoff triple is said to be \emph{proper} if $a,b$ and $c$ are distinct. Any integer appearing in a Markoff triple is called a \emph{Markoff number}.

We define two maps $\mathbb{Z}^3\to\mathbb{Z}^3$ by
\[ m_L : (a,b,c) \mapsto (b,3bc-a,c) \quad\text{and}\quad
   m_R : (a,b,c) \mapsto (a,3ab-c,b) \]
It is well-known and easy to verify that the image of a Markoff triple under $m_L$ or $m_R$ is once again a Markoff triple.

Let $\mathcal{M}$ be the set of triples $(a,b,c)$ generated by iterative applications of $m_L$ and $m_R$ starting from $(1,5,2)$. 
\[ \mathcal{M} = \set{(1,5,2)} \cup \set{ M \;\left|\; (1,5,2) \stackrel{m_1}{\lto}\cdots\stackrel{m_n}{\lto}M \text{ where } m_i = m_L \text{ or } m_i = m_R \right.}\]
We can then view $m_L$ and $m_R$ as maps $\mathcal{M}\to\mathcal{M}$. By induction it is easy to see that every $(a,b,c)\in\mathcal{M}$ satisfies $a<b$, $c<b$ and $a\neq c$.
By adapting a classical result due to Markoff~\cite{M2}, one can show that $\mathcal{M}$ coincides with the set of proper Markoff triples.
Furthermore, we define a third function $m_C:\mathcal{M}\backslash\{(1,5,2)\}\to\mathcal{M}$ by
\[ m_C : (a,b,c) \mapsto \left\{ \begin{array}{ll}  (3ac-b,a,c) & \mbox{if $a>c$} \\
                                                    (a,c,3ac-b) & \mbox{if $a<c$} \end{array} \right. \]
A straightforward calculation shows that $m_Cm_L = \mathbb{I}_{\mathcal{M}}$ and $m_Cm_R = \mathbb{I}_{\mathcal{M}}$.
Thus $m_L$ and $m_R$ are injective.
\begin{proposition}
The triple $(\mathcal{M}, m_L, m_R)$ is a binary tree rooted in $(1,5,2)$.
\end{proposition}
\begin{proof}
It is easy to see that there is no proper Markoff triple $M$ such that $m_L(M)=(1,5,2)$ or $m_R(M)=(1,5,2)$. 
Moreover, if $(a,b,c)\in\Img m_L$, then $a>c$ and if $(a,b,c)\in\Img m_R$, then $a<c$. Thus $\Img m_L \cap \Img m_R = \emptyset$.
By the definition of $\mathcal{M}$, it is then clear that $(\mathcal{M}, m_L, m_R)$ is a binary tree rooted in $(1,5,2)$.
\end{proof}

Much of the work relating to Markoff numbers has been motivated by the uniqueness conjecture first formulated by Frobenius in 1913.
\begin{conjecture}[Frobenius \cite{Fro}]
Every Markoff triple is uniquely determined by its largest term.
\end{conjecture}
This conjecture is known to hold in several special cases. See, for instance, \cite{Ba}, \cite{BRS}, \cite{Bu}, \cite{Z}.

\subsection{Christoffel Words} 
A \emph{lattice path} is a path in $\mathbb{R}^2$ composed of consecutive line segments of the form $[(a,b),(a+1,b)]$ or $[(a,b),(a,b+1)]$ with $a,b\in\mathbb{Z}$. 

Let $p$ and $q$ be relatively prime non-negative integers. The \emph{Christoffel path of slope $q/p$} is the lattice path from $(0,0)$ to $(p,q)$ lying weakly below the line segment $[(0,0),(p,q)]$ such that the region bounded by these two paths contains no points of $\mathbb{Z}\times\mathbb{Z}$.

The \emph{Christoffel word of slope $q/p$} is the word $C(p,q) \in \set{x,y}^*$ obtained by following the Christoffel path of slope $q/p$ starting from $(0,0)$ and encoding each segment of the form $[(a,b),(a+1,b)]$ by $x$ and each segment of the form $[(a,b),(a,b+1)]$ by $y$. All Christoffel words different from $x$ and $y$ are called \emph{proper}.

Borel and Laubie \cite{BL} showed that every proper Christoffel word can be uniquely expressed as the concatenation of two Christoffel words; this is called the \emph{standard factorisation}. Given a Christoffel word $C(p,q)$, there is a unique integer point $(c,d)$ on the corresponding Christoffel path having minimum nonzero distance to the line segment $[(0,0),(p,q)]$. The standard factorisation of $C(p,q)$ is then given by $C(p,q) = C(c,d)C(p-c,q-d)$.

\begin{proposition}[{\cite[Proposition 1]{BL}}]\label{cwcon}
The concatenation $C(p_1,q_1)C(p_2,q_2)$ is a Christoffel word if and only if $\det \left[\begin{smallmatrix} p_1&q_1\\ p_2&q_2 \end{smallmatrix}\right] = 1$.
\end{proposition}

It is often convenient to use the tree construction of Christoffel words (see \cite{BLRS}, \cite{BdL},\cite{BL}).
Let $\mathcal{C}$ be the set of triples $(w_1,w_2,w_3)$ where $w_2$ is a Christoffel word and $w_2= w_1w_3$ is the standard factorisation of $w_2$.
An element of $\mathcal{C}$ will be called a \emph{Christoffel triple}. 
We define two maps $c_L, c_R: \mathcal{C} \to \mathcal{C}$ by $c_L(w_1,w_2,w_3) = (w_2, w_2w_3, w_3)$ and $c_R(w_1,w_2,w_3) = (w_1, w_1w_2, w_2)$. Then $(\mathcal{C}, c_L, c_R)$ is a binary tree with root $(x,xy,y)$. 

A comprehensive overview of the theory of Christoffel words can be found in \cite{BLRS}.

\subsection{String Modules} 
We recall some basic facts about string modules; we refer to \cite{BR} for further details.

Let $Q=(Q_0,Q_1)$ be a quiver with $Q_0$ the set of vertices and $Q_1$ the set of arrows.
For an arrow $\alpha:i\to j$, let $s(\alpha)=i$ denote its source and $t(\alpha)=j$ its target.
Given $\alpha \in Q_1$, we denote by $\alpha^{-1}$ its formal inverse where $s(\alpha^{-1}) = t(\alpha)$ and $t(\alpha^{-1}) = s(\alpha)$.
We use $Q_1^{-1}$ to represent the set of formal inverses of arrows in $Q_1$. 
Note that $(\alpha^{-1})^{-1} = \alpha$.

A \emph{string} of length $n\geq 1$ for a bound quiver $(Q,I)$ is a sequence $w=a_1a_2\cdots a_n$ of elements $a_i\in Q_1 \cup Q_1^{-1}$ 
satisfying:
\begin{enumerate}[label=(\arabic{*})]
  \item $t(a_i) = s(a_{i+1})$ for $1\leq i \leq n-1$.
  \item $a_i \neq a_{i+1}^{-1}$ for $1\leq i \leq n-1$.
  \item For every subsequence $v$ of $w$, neither $v$ nor $v^{-1}$ is contained in $I$.
\end{enumerate}
Additionally, for every vertex $i\in Q_0$, we define a string $\varepsilon_i$ of length 0 such that $s(\varepsilon_i)=t(\varepsilon_i)=i$.

Let $k$ be an algebraically closed field. A $k$-algebra $A=kQ/I$ is a \emph{string algebra} if the following hold:
\begin{enumerate}[label=(\arabic{*})]
  \item The ideal $I$ is generated by monomial relations.
  \item Each vertex of $Q$ is the source of at most two arrows and the target of at most two arrows.
  \item For every arrow $\beta$, there is at most one arrow $\alpha$ such that $\alpha\beta \notin I$ and at most one arrow $\gamma$ such that 
  $\beta\gamma \notin I$.
\end{enumerate}
Let $A=kQ/I$ be a string algebra. For a given string $w$ it is possible to construct an indecomposable $A$-module $M(w)$ called a \emph{string module} (see \cite{BR}). 

By a result due to Crawley-Boevey \cite{CB} we can explicitly describe a basis for $\Hom_A(M(w_1),M(w_2))$ where $w_1$ and $w_2$ are strings.
A string $v$ is a \emph{factor string} of $w$ if $w=xvy$ where $x,y$ are strings such that $x$ ends with an inverse arrow or is of length 0 and $y$ starts with an arrow or is of length 0.
A string $v$ is a \emph{substring} of $w$ if $w=xvy$ where $x,y$ are strings such that $x$ ends with an arrow or is of length 0 and $y$ starts with an inverse arrow or is of length 0.
For two strings $w_1$ and $w_2$, a pair $a = ((x_1,v_1,y_1),(x_2,v_2,y_2))$ is \emph{admissible} if $w_1=x_1v_1y_1$ where $v_1$ is a factor string of $w_1$, $w_2=x_2v_2y_2$ where $v_2$ is a substring of $w_2$ and $v_1=v_2$ or $v_1=v_2^{-1}$.
For each admissible pair one can define a morphism $f_a : M(w_1)\to M(w_2)$. Moreover, the set of such maps forms a basis of $\Hom_A(M(w_1),M(w_2))$.

In the special case where $w_2$ is a factor string of $w_1$, we will simply write $w_1 \twoheadrightarrow w_2$ instead of $f_a$ where $a=((x,w_2,y),(\varepsilon_i,w_2,\varepsilon_j))$ when there is no risk of confusion.

\section{Exact Sequences with minimal approximations}

In this section we will study the properties of short exact sequences having a minimal approximation as one of its morphisms.

Let $\D$ be an additive subcategory of an additive category $\C$.
Recall that a morphism $g:X\to Y$ is called \emph{right minimal} if every endomorphism $h:X\to X$ such that $g=gh$ is an automorphism.
Let $Y$ be an object in $\C$. A morphism $g:X\to Y$ with $X$ in $\D$ is a \emph{right $\D$-approximation} if the induced map
$g_* = \HomC(Z,g):\HomC(Z,X) \to \HomC(Z,Y)$ is an epimorphism for every object $Z$ in $\D$.
Note that $g$ is a right $\D$-approximation if and only if every map in $\HomC(Z,Y)$ factors through $g$, for every $Z \in \D$.
Left minimal morphisms and left $\D$-approximations are defined dually.

\begin{lemma}[{\cite[Corollary 1.4]{KS}}]\label{mindec}
Let $f:X\to Y$ be a morphism in a Krull-Schmidt category.
\begin{enumerate}
  \item There is a decomposition $f = \smatmapv{f_1}{f_2} : X \to Y_1 \oplus Y_2 = Y$ such that $f_1$ is left minimal and $f_2 = 0$.
  \item There is a decomposition $f = \smatmaph{f_1}{f_2} : X=X_1\oplus X_2 \to Y$ such that $f_1$ is right minimal and $f_2 = 0$.
\end{enumerate}
\end{lemma}

Throughout this section $\D$ will denote an additive subcategory of an abelian Krull-Schmidt category $\C$ such that
$\Ext^1_{\C}(\D,\D) = 0$. We fix an indecomposable object $M$ in $\D$ and set
$\overline{\D} = \add{\ind \D \backslash \set{M}}$.

The following result is a direct adaptation of Lemmas~6.4 and 6.6 of~\cite{BMRRT}.
\begin{lemma}\label{sesR1}
If
\begin{equation}
\SES{M'}{f}{E}{g}{M}\label{ses}
\end{equation}
is a short exact sequence in $\C$ with $g$ a minimal right $\overline{\D}$-approximation, then
\begin{enumerate}[ref=\cref{sesR1}(\alph{*})]
  \item The morphism $f$ is a minimal left $\overline{\D}$-approximation. \label{fmin}
  \item The object $M'$ is not in $\D$.                                   \label{m'notinD}
  \item $\Ext^1_{\C}(\overline{\D}, M')=0$.                               \label{extdm'}
  \item If $\Ext^1_{\C}(M',E)=0$ then $\Ext^1_{\C}(M',M')=0$.             \label{extm'm'}
\end{enumerate}
\end{lemma}
\begin{proof}\mbox{}
\begin{enumerate}
  \item Let $X \in \overline{\D}$. Applying the functor $\HomC(-,X)$ to (\ref{ses})
        we get the exact sequence
        \[ \HomC(E,X)\stackrel{f^*}{\lto}\HomC(M',X) \lto \Ext^1_{\C}(M,X) \]
        We have $\Ext^1_{\C}(M,X) = 0$ since $M, X \in \D$.
        Thus, $f^*$ is an epimorphism.

        Now suppose that $f$ is not minimal. By \cref{mindec}, there is a non-trivial decomposition
        $f =\smatmapv{f_1}{f_2} : M' \to E_1 \oplus E_2$ with $f_1$ minimal and $f_2 = 0$.
        Hence $E_2$ is isomorphic to a direct summand of $M$. But $M$ indecomposable implies that
        $M\cong E_2$, contradicting the fact that $E_2 \in \overline{\D}$.

  \item Suppose that $M' \in \D$. We then have $\Ext^1_{\C}(M,M') = 0$ and it follows that the exact sequence 
        (\ref{ses}) splits. Thus $E \cong M'\oplus M$, but $E\in \overline{\D}$ implies that $M\in \overline{\D}$.
 
  \item Let $X \in \overline{\D}$. Applying $\HomC(X,-)$ to (\ref{ses}), we get the exact sequence
        \[ \HomC(X,E)\stackrel{g_*}{\lto} \HomC(X,M) \lto \Ext^1_{\C}(X,M') \lto \Ext^1_{\C}(X,E) \]
        where $\Ext^1_{\mathcal{C}}(X,E) = 0$ since $X,E \in \overline{\D} \subseteq \D$.
        Moreover, $g$ is a right $\overline{\D}$-approximation and so the induced morphism $g_*$ is an epimorphism.
        Hence $\Ext^1_{\C}(X, M')=0$
  
  \item By applying $\HomC(M',-)$ to (\ref{ses}) we get the exact sequence
        \[ \HomC(M',E)\stackrel{g_*}{\lto}\HomC(M',M) \lto \Ext^1_{\C}(M',M') \lto 0 \]
        To prove that $\Ext^1_{\C}(M', M')=0$ it suffices to show that $g_*$ is an epimorphism.
        Let $h \in \HomC(M',M)$. We will show that $h$ factors through $g$.
        We have an exact sequence 
        \[ \HomC(E,M)\stackrel{f^*}{\lto}\HomC(M',M) \lto \Ext^1_{\C}(M,M) \]
        Thus $f^*$ is an epimorphism since $\Ext^1_{\C}(M,M)=0$. So there exists
        $t \in \HomC(E,M)$ such that $h=tf$. Since $g$ is a right $\overline{\D}$-approximation,
        there exists $s\in \HomC(E,E)$ such that $t=gs$. Therefore $h=tf=g(sf)$.
\end{enumerate}
\end{proof}

We will now consider the subcategory $\D' = \add{\set{M'} \cup \overline{\D}}$
of $\C$ obtained from $\D$ by ``replacing'' $M$ by $M'$. By \ref{m'notinD},
$\D$ and $\D'$ must be distinct subcategories of $\C$.

\begin{proposition}\label{sesR2}
Let \[ \SES{M'}{f}{E}{g}{M}\]
be a short exact sequence in $\C$ where $g$ is a minimal right $\overline{\D}$-approximation.
For every $n>1$, if $\Ext^i_{\C}(\D,\D)=0$ for all $1\leq i \leq n$, then
$\Ext^{i}_{\C}(\D',\D')=0$ for all $1\leq i \leq n-1$.
\end{proposition}
\begin{proof}
Suppose that $\Ext^i_{\C}(\D,\D)=0$ for all $1\leq i \leq n$.
It is sufficient to show that $\Ext^{i}_{\C}(M',M')=0$, $\Ext^{i}_{\C}(M',\overline{\D})=0$
and $\Ext^{i}_{\C}(\overline{\D}, M')=0$ for all $1\leq i \leq n-1$.

Let $X\in\overline{\D}$. We have an exact sequence
\[\Ext^{i}_{\C}(E,X) \lto \Ext^{i}_{\C}(M',X) \lto \Ext^{i+1}_{\C}(M,X)\]
Because $E,X,M\in\D$ we get $\Ext^{i}_{\C}(E,X)=0$ and $\Ext^{i+1}_{\C}(M,X)=0$ since $i+1\leq n$.
Thus $\Ext^{i}_{\C}(M',X)=0$.

Next, when $i\geq 2$, the sequence
\[\Ext^{i-1}_{\C}(X,M) \lto \Ext^{i}_{\C}(X,M') \lto \Ext^{i}_{\C}(X,E)\]
is exact. Since $E,X,M\in\D$, we get $\Ext^{i-1}_{\C}(X,M)=0$ and
$\Ext^{i}_{\C}(X,E)=0$. When $i=1$, \ref{extdm'} applies. In both cases
$\Ext^{i}_{\C}(X,M')=0$.

Finally, when $i\geq 2$, consider the exact sequence
\[\Ext^{i-1}_{\C}(M',M) \lto \Ext^{i}_{\C}(M',M')\lto \Ext^{i}_{\C}(M',E) = 0\]
To establish that $\Ext^{i}_{\mathcal C}(M',M')=0$, it is sufficient to check
$\Ext^{i-1}_{\mathcal C}(M',M)=0$, but this is a consequence of
\[0=\Ext^{i-1}_{\C}(E,M) \lto \Ext^{i-1}_{\C}(M',M)\lto \Ext^{i}_{\C}(M,M) = 0\]
When $i=1$, we can apply \ref{extm'm'} since $\Ext^{1}_{\C}(M',E)=0$.
We conclude that $\Ext^{i}_{\C}(M',M')=0$.
\end{proof}

\begin{corollary}\label{sesR3}
If $\Ext^i_{\C}(\D,\D)=0$ for all $i \geq 1$, then
$\Ext^{i}_{\C}(\D',\D')=0$ for all $i \geq 1$.
\end{corollary}

Starting with a monomorphic minimal left approximation, we obtain results dual to those presented previously in this section. They are summarised in the following proposition:
\begin{proposition}
Let \[ \SES{M}{f}{E}{g}{M''}\]
be a short exact sequence in $\C$ where $f$ is a minimal left $\overline{\D}$-approximation.
Let $\D'' = \add{\set{M''}} \cup \overline{\D})$. Then
\begin{enumerate}
  \item The morphism $g$ is a minimal right $\overline{\D}$-approximation.
  \item The object $M''$ is not in $\D$.
  \item $\Ext^1_{\C}(M'', \overline{\D})=0$.
  \item If $\Ext^1_{\C}(E,M'')=0$ then $\Ext^1_{\C}(M'',M'')=0$.
  \item If $\Ext^i_{\C}(\D,\D)=0$ for all $1\leq i \leq n$, then
           $\Ext^{i}_{\C}(\D'',\D'')=0$ for all $1\leq i \leq n-1$.
  \item If $\Ext^i_{\C}(\D,\D)=0$ for all $i \geq 1$, then
           $\Ext^{i}_{\C}(\D'',\D'')=0$ for all $i \geq 1$.
\end{enumerate}
\end{proposition}

\section{Markoff Modules}

Let $\mathcal{C}$ be an abelian Krull-Schmidt $k$-category for some field $k$.
A triple $(M_1, M_2, M_3)$ of pairwise non-isomorphic indecomposable objects in $\mathcal{C}$
is called \emph{mutable} if it satisfies the following conditions:

\begin{enumerate}[label=(M\arabic{*}), ref=(M\arabic{*})]
\item $\Ext^n_{\C}(M_1\oplus M_2\oplus M_3, M_1\oplus
M_2\oplus M_3)=0$ for all $n\geq 1$. \label{M1}

\item $\dim_k \End_{\C}(M_i)=1$ for $i=1,2,3$. \label{M2}

\item $\HomC(M_i,M_j) = 0$ whenever $i>j$.\label{M3}

\item $\HomC(M_1,M_2)$ has a basis $\set{\beta_1, \beta_2}$ with $\beta_1$ and $\beta_2$ monomorphisms,\\
$\HomC(M_2,M_3)$ has a basis $\set{\alpha_1, \alpha_2}$ with $\alpha_1$ and $\alpha_2$ epimorphisms,\\
$\HomC(M_1,M_3)$ has a basis $\set{\gamma_1, \gamma_2}$ such that $\gamma_1=\alpha_1\beta_1$ and $\gamma_2=\alpha_2\beta_2$,\\
$\alpha_1\beta_2=0$ and $\alpha_2\beta_1=0$.\label{M4}
\end{enumerate}
In particular the quiver of $\End_{\mathcal{C}}(M_1\oplus M_2\oplus M_3)$ is of the form
\begin{equation*}
\xymatrix{ M_1 \ar@{.}@/^1pc/[rr]\ar@{.}@/_1pc/[rr]
\ar@<2pt>[r]\ar@<-2pt>[r] & M_2 \ar@<2pt>[r]\ar@<-2pt>[r] & M_3}
\end{equation*}
We remark that it is not necessary to assume that $M_1, M_2, M_3$ are indecomposable and pairwise non-isomorphic as these properties are consequences of \ref{M2} and \ref{M3}.

\begin{proposition}
If $(M_1, M_2, M_3)$ is a mutable triple, then
\begin{enumerate}
    \item $M_1$ is not injective.
    \item $M_3$ is not projective.
    \item $M_2$ is neither injective, nor projective.
\end{enumerate}
\end{proposition}
\begin{proof}\mbox{}
\begin{enumerate}
    \item Dual of (b).
    \item Suppose that $M_3$ is projective. The epimorphism 
          $\alpha_1:M_2 \to M_3$ must then be a retraction.
          Thus $M_3$ is isomorphic to a direct summand of $M_2$, which leads to
          the contradiction $M_3\cong M_2$.
    \item Suppose that $M_2$ is projective. Since $\alpha_1$ is an epimorphism,
          there exists $h:M_2 \to M_2$ such that $\alpha_1 h = \alpha_2$.
          But $h \in \End(M_2) \cong_k k$, so $h=\lambda\mathbb{I}_{M_2}$ for some $\lambda \in k$.
          Whence the contradiction $\alpha_2 = \lambda \alpha_1$.
          The non-injectivity of $M_2$ is shown analogously.
\end{enumerate}
\end{proof}

\begin{proposition}\label{mtses}
Let $(M_1, M_2, M_3)$ be a mutable triple. There are short exact sequences in $\C$
\begin{enumerate}
    \item \begin{equation}\label{mmses}
          \SES{M'_3}{f}{M_2\oplus M_2}{g}{M_3}
          \end{equation}
where $f$ is a minimal left $\add{M_1\oplus M_2}$-approximation and $g$ is a minimal right $\add{M_1\oplus M_2}$-approximation.
    \item
\[ \SES{M_1}{f'}{M_2\oplus M_2}{g'}{M'_1} \]
where $f'$ is a minimal left $\add{M_2\oplus M_3}$-approximation and $g'$ is a minimal right $\add{M_2\oplus M_3}$-approximation.
\end{enumerate}
\end{proposition}
\begin{proof}
We will only prove (a), the proof of (b) is similar.
The map $g = \matmaph{\alpha_1}{\alpha_2}:M_2 \oplus M_2 \to M_3$ is an epimorphism by condition \ref{M4}.
Let $M'_3 = \Ker g$, then the sequence (\ref{mmses}) is exact.
All maps $M_1\to M_3$ and $M_2\to M_3$ factor through $g$ by \ref{M4} from which we deduce that $g$ is a right $\add{M_1\oplus M_2}$-approximation. The minimality of $g$ follows from \cref{mindec}.
Finally, by \ref{fmin}, $f$ is a minimal left $\add{M_1\oplus M_2}$-approximation.
\end{proof}

Let $T=(M_1, M_2, M_3)$ be a mutable triple. We define $\mu_L(T) = (M_2, M'_1, M_3)$ and $\mu_R(T) = (M_1, M'_3, M_2)$. We will show that $\mu_L(T)$ and $\mu_R(T)$ are mutable triples.

\begin{lemma}\label{homext}
If $\mu_R(T) = (M_1, M'_3, M_2)$ where $T=(M_1,M_2,M_3)$ is a mutable triple, then
\begin{enumerate}[ref=(\alph{*})]
  \item $\dim_k \HomC(M_1, M'_3) = \dim_k \HomC(M'_3, M_2) = 2$             \label{homm'3m2}
  \item $\HomC(M'_3, M_1) = \HomC(M_2,M'_3) = 0$                            \label{homm2m'3}
  \item $\dim_k \HomC(M'_3, M_3) = 3$ and $\Hom_{\mathcal{C}}(M_3, M'_3) =0$
  \item $\dim_k \Ext^1_{\C}(M_3,M'_3) = 1$                                  \label{extm3m'3}
  \item $\dim_k \End_{\C}(M'_3) = 1$
\end{enumerate}
\end{lemma}
\begin{proof}\mbox{}
\begin{enumerate}
  \item The sequence
        \[ \SES{\HomC(M_1,M'_3)}{}{\HomC(M_1,M_2\oplus M_2)}{g_*}{\HomC(M_1,M_3)} \]
        is exact since $g$ is a right $\add{M_1\oplus M_2}$-approximation.
        We know that $\dim_k \HomC(M_1,M_3)=2$ and $\dim_k \HomC(M_1,M_2\oplus M_2)=4$. From the additivity of $\dim_k$
        on short exact sequences, we have $\dim_k \HomC(M_1,M'_3)=2$.

        The sequence
        \[ \SES{\HomC(M_3,M_2)}{}{\HomC(M_2\oplus M_2,M_2)}{f^*}{\HomC(M'_3,M_2)} \]
        is exact since $f$ is a left $\add{M_1\oplus M_2}$-approximation.
        We have $\dim_k \HomC(M_2\oplus M_2,M_2) = 2$ and $\dim_k \HomC(M_3, M_2) = 0$.
        Thus $\dim_k \Hom_{\mathcal{C}}(M'_3, M_2) = 2$.
           
  \item The induced morphism $f^* : \HomC(M_2\oplus M_2,M_1) \to \HomC(M'_3,M_1)$ is an epimorphism and
        $\HomC(M_2\oplus M_2, M_1)=0$.

        The sequence 
        \[ \SES{\HomC(M_2,M'_3)}{}{\HomC(M_2,M_2\oplus M_2)}{g_*}{\HomC(M_2,M_3)} \]
        is exact since $g$ is a right $\add{M_1\oplus M_2}$-approximation. We have
        $\dim_k \HomC(M_2,M_2\oplus M_2) = 2$ and $\dim_k \HomC(M_2, M_3) = 2$.
        Thus $\dim_k \HomC(M_2,M'_3) = 0$
  
  \item We have an exact sequence
        \begin{align*}
          0\lto\HomC(M_3,M_3)\lto\HomC(M_2\oplus M_2,M_3) &\lto \HomC(M'_3,M_3) \\
                                                          &\lto \Ext^1_{\C}(M_3,M_3) = 0
        \end{align*}
        where $\dim_k \HomC(M_3,M_3)=1$ and $\dim_k \HomC(M_2\oplus M_2,M_3)=4$.

        The induced morphism $g^* : \HomC(M_3, M'_3) \to\HomC(M_2\oplus M_2,M'_3)$ is a monomorphism since
        the functor $\HomC(-,M'_3)$ is left exact and, by \ref{homm2m'3}, we get
        $\HomC(M_2\oplus M_2,M'_3)=0$.
  
  \item The sequence
        \begin{align*}
          \HomC(M_3,M_2\oplus M_2)\lto \HomC(M_3,M_3) &\lto \Ext^1_{\C}(M_3,M'_3) \\
                                                      &\lto \Ext^1_{\C}(M_3,M_2\oplus M_2)
        \end{align*}
        is exact. We have $\HomC(M_3,M_2\oplus M_2)=0$, $\Ext^1_{\C}(M_3,M_2\oplus M_2)=0$ and
        $\dim_k \HomC(M_3,M_3)=1$, from which we deduce
        $\dim_k \Ext^1_{\C}(M_3,M'_3)=1$.
  
  \item We have an exact sequence
        \begin{align*}
          \HomC(M_2\oplus M_2,M'_3)\lto \HomC(M'_3,M'_3) &\lto \Ext^1_{\C}(M_3,M'_3) \\
                                                         &\lto \Ext^1_{\C}(M_2\oplus M_2,M'_3)
        \end{align*}
        By \ref{homm2m'3}, $\HomC(M_2\oplus M_2,M'_3)=0$ and it follows from \ref{extdm'} that
        $\Ext^1_{\C}(M_2\oplus M_2,M'_3)=0$.
        From \ref{extm3m'3} we get $\dim_k \Ext^1_{\C}(M_3,M'_3)=1$, hence
        $\dim_k \HomC(M'_3,M'_3)=1$
\end{enumerate}
\end{proof}

\begin{lemma}\label{homextD}
If $\mu_L(T) = (M_2, M'_1, M_3)$ where $T=(M_1,M_2,M_3)$ is a mutable triple, then
\begin{enumerate}
  \item $\dim_k \HomC(M_2, M'_1) = \dim_k \HomC(M'_1, M_3) = 2$
  \item $\HomC(M'_1, M_2) = \HomC(M_3,M'_1) = 0$
  \item $\dim_k \HomC(M'_1, M_1) = 0$ and $\Hom_{\mathcal{C}}(M_1, M'_1) =3$
  \item $\dim_k \Ext^1_{\C}(M'_1,M_1) = 1$
  \item $\dim_k \End_{\C}(M'_1) = 1$
\end{enumerate}
\end{lemma}

\begin{lemma}\label{m'3bases}
Let $T=(M_1,M_2,M_3)$ be a mutable triple.
\begin{enumerate}
  \item If $\mu_R(T) = (M_1, M'_3, M_2)$, then there are bases $\set{\beta'_1, \beta'_2}$ and $\set{\alpha'_1, \alpha'_2}$
of $\HomC(M_1, M'_3)$ and $\HomC(M'_3,M_2)$ respectively, such that $\beta'_1$ and
$\beta'_2$ are monomorphisms, $\alpha'_1$ and $\alpha'_2$ are epimorphisms, 
$\beta_1 = \alpha'_2 \beta'_2$, $\beta_2 = \alpha'_1 \beta'_1$, $\alpha'_2 \beta'_1 = 0 = \alpha'_1 \beta'_2$.
  \item If $\mu_L(T) = (M_2, M'_1, M_3)$, then there are bases $\set{\beta'_1, \beta'_2}$ and $\set{\alpha'_1, \alpha'_2}$
of $\HomC(M_2, M'_1)$ and $\HomC(M'_1,M_3)$ respectively, such that $\beta'_1$ and
$\beta'_2$ are monomorphisms, $\alpha'_1$ and $\alpha'_2$ are epimorphisms, 
$\alpha_1 = \alpha'_2 \beta'_2$, $\alpha_2 = \alpha'_1 \beta'_1$, $\alpha'_2 \beta'_1 = 0 = \alpha'_1 \beta'_2$.
\end{enumerate}
\end{lemma}
\begin{proof}
We will only prove (a), the proof of (b) is similar.

Let $\alpha'_1, \alpha'_2: M'_3\to M_2$ be the components of the morphism $f$ appearing in the short exact sequence (\ref{mmses}).
As seen in the proof of \cref{homext}\ref{homm'3m2}, $f^* : \HomC(M_2\oplus M_2,M_2) \to \HomC(M'_3,M_2)$ is an isomorphism. Let $\pi_1, \pi_2$ be the canonical projections $M_2\oplus M_2 \twoheadrightarrow M_2$, then it follows from \ref{M2} that $\set{\pi_1,\pi_2}$ is a basis of $\HomC(M_2\oplus M_2,M_2)$. Thus $\set{\alpha'_1, \alpha'_2} = f^*(\set{\pi_1,\pi_2})$ is a basis of $\HomC(M'_3,M_2)$. 
We must now show that $\alpha'_1$ and $\alpha'_2$ are epimorphisms. Consider the following commutative diagram with exact rows:
\[ \xymatrix{
0 \ar[r] & M'_3 \ar[r]^-{f} \ar[d]^{\alpha'_1} & M_2\oplus M_2 \ar[r]^-{g} \ar[d]^{\pi_1} & M_3 \ar[r] \ar[d] & 0 \\
0 \ar[r] & M_2  \ar@{=}[r]                    & M_2           \ar[r]                    & 0                 &   \\
} \]
By the snake lemma, there is an exact sequence
\[ M_2 \stackrel{\alpha_2}{\lto} M_3 \lto \Coker \alpha'_1 \lto 0\]
and so $\Coker \alpha'_1 = 0$ since $\alpha_2$ is an epimorphism. The same argument shows that $\alpha'_2$ is an epimorphism. 

We have an exact sequence
\[ \SES{\HomC(M_1,M'_3)}{f_*}{\HomC(M_1,M_2\oplus M_2)}{g_*}{\HomC(M_1,M_3)} \]
thus $\HomC(M_1,M'_3) \cong \Img f_* = \Ker g_*$ and so $\dim_k \Img f_* = 2$.
By \ref{M4}, $\alpha_1\beta_2=0$ and $\alpha_2\beta_1=0$ hence
$\smatmapv{\beta_2}{0}$ and $\smatmapv{0}{\beta_1} \in \Ker g_*$ since $g=\smatmaph{\alpha_1}{\alpha_2}$.
These two morphisms form a basis of $\Img f_*$ and consequently there exists a basis $\set{\beta'_1, \beta'_2}$
of $\HomC(M_1, M'_3)$ satisfying
\[ f \beta'_1 = \matmapv{\alpha'_1 \beta'_1}{\alpha'_2 \beta'_1} = \matmapv{\beta_2}{0}
\quad\text{ and }\quad f \beta'_2 = \matmapv{\alpha'_1 \beta'_2}{\alpha'_2 \beta'_2} = \matmapv{0}{\beta_1} \]
Furthermore, since $\beta_1$ and $\beta_2$ are monomorphisms, $\beta'_1$ and $\beta'_2$ must also be monomorphisms.
\end{proof}

By combining the results from \cref{sesR3,homext,homextD,m'3bases}, we get the following:
\begin{theorem}
If $T=(M_1, M_2, M_3)$ is a mutable triple, then so are $\mu_L(T)$ and $\mu_R(T)$.
\end{theorem}

Let $\mathcal{T}(T_0)$ be the set of all mutable triples obtained by iterative applications of $\mu_L$ and $\mu_R$ starting from some triple $T_0$.
\[ \mathcal{T}(T_0) = \{T_0\} \cup \left\{ T \;\left|\; T_0 \stackrel{\mu_1}{\lto}\cdots\stackrel{\mu_n}{\lto}T \text{ where } \mu_i = \mu_L \text{ or } \mu_i = \mu_R \right.\right\}\]
The two operations $\mu_L$ and $\mu_R$ can then be seen as functions $\mathcal{T}(T_0)\to\mathcal{T}(T_0)$. A triple $T$ is said to be \emph{non-initial} if $T\in\mathcal{T}(T_0)\backslash\{T_0\}$ for some $T_0$.

\begin{lemma}\label{mmdichotomy}
A non-initial mutable triple $T = (M_1, M_2, M_3)$ satisfies the dichotomy:
\[ \matmapv{\alpha_1}{\alpha_2} : M_2 \to M_3\oplus M_3 \text{ is a monomorphism}\]
or
\[ \matmaph{\beta_1}{\beta_2} : M_1\oplus M_1 \to M_2 \text{ is an epimorphism}\]
The first case occurs when $T=\mu_R(T')$ and the second when $T=\mu_L(T')$ for some mutable triple $T'$.
\end{lemma}
\begin{proof}
We will only consider the first case, the second being similar.

Let $T = (M_1, M_2, M_3)$ be a mutable triple and $\mu_R(T) = (M_1, M'_3, M_2)$. Then there is an exact sequence
\[ \SES{M'_3}{\smatmapv{\alpha'_1}{\alpha'_2}}{M_2\oplus M_2}{}{M_3} \]
In particular $\smatmapv{\alpha'_1}{\alpha'_2}$ is a monomorphism.

We have $\HomC(M'_3,M_2) = \langle \beta'_1, \beta'_2\rangle$ where the $\beta'_i$ are as in \cref{m'3bases}.
Now suppose that $\smatmaph{\beta'_1}{\beta'_2}$ is an epimorphism.
\[ \matmapv{\alpha'_1}{\alpha'_2}\matmaph{\beta'_1}{\beta'_2} = \begin{bmatrix} \alpha'_1\beta'_1&\alpha'_1\beta'_2\\ \alpha'_2\beta'_1&\alpha'_2\beta'_2 \end{bmatrix} = \begin{bmatrix} \beta_2&0\\ 0&\beta_1 \end{bmatrix} \]
Since $\beta_1$ and $\beta_2$ are monomorphisms, we obtain the contradiction: $\matmaph{\beta'_1}{\beta'_2} : M_1\oplus M_1 \to M'_3$ is an isomorphism.
\end{proof}

This leads us to define $\mu_C:\mathcal{T}(T_0)\backslash\{T_0\}\to\mathcal{T}(T_0)$ by
\[ \mu_C(M_1, M_2, M_3)= \left\{ \begin{array}{ll} (M_1, M_3, \Coker\smatmapv{\alpha_1}{\alpha_2}) & \mbox{ if $\smatmapv{\alpha_1}{\alpha_2}$ is a monomorphism} \\
                                     (\Ker\smatmaph{\beta_1}{\beta_2},M_1,M_3) & \mbox{ if $\smatmaph{\beta_1}{\beta_2}$ is an epimorphism}
                                 \end{array} \right. \]
The operations $\mu_L$, $\mu_C$ and $\mu_R$ can be seen as analogues of the mutations of a tilting object in a cluster category as defined in \cite{BMRRT}.
These maps satisfy $\mu_C\mu_L = \mathbb{I}_{\mathcal{T}(T_0)}$ and $\mu_C\mu_R = \mathbb{I}_{\mathcal{T}(T_0)}$ as will be shown in the following proposition:

\begin{proposition}
The triple $(\mathcal{T}(T_0), \mu_L, \mu_R)$ is a binary tree rooted in $T_0$.
\end{proposition}
\begin{proof}
By \cref{mmdichotomy}, $\Img\mu_L\cap\Img\mu_R = \emptyset$.
To establish that $\mu_R$ is injective, it is sufficient to show that $\mu_C\mu_R = \mathbb{I}_{\mathcal{T}(T_0)}$.
Let $T=(M_1, M_2, M_3) \in \mathcal{T}(T_0)$, then $\mu_R(T)=T'=(M_1, M'_3, M_2)$ with 
$\HomC(M'_3,M_2) = \langle\alpha'_1, \alpha'_2\rangle$ as in \cref{m'3bases}. By \cref{mmdichotomy}, $\smatmapv{\alpha'_1}{\alpha'_2}$ is a monomorphism. Thus $\mu_C(T') = (M_1, M_2, \Coker\smatmapv{\alpha'_1}{\alpha'_2})$. But since $T'=\mu_R(T)$ there is an exact sequence
\[ \SES{M'_3}{\smatmapv{\alpha'_1}{\alpha'_2}}{M_2\oplus M_2}{}{M_3} \]
From which we conclude that $\mu_C(\mu_R(T)) = T$. The injectivity $\mu_L$ is shown similarly.
\end{proof}

Consider the string algebra given by the quiver
\begin{equation*}
\xymatrix{ 2 \ar@<2pt>[r]^{\alpha} \ar@<-2pt>[r]_{\gamma} & 1
\ar@<2pt>[r]^{\beta} \ar@<-2pt>[r]_{\delta} & 3}
\end{equation*} bound by $\alpha\beta = 0$ and
$\gamma\delta=0$. 
Consider the strings $w_1 = \varepsilon_1$, $w_2 = \alpha^{-1}\gamma\beta\delta^{-1}\alpha^{-1}\gamma$ and $w_3 = \alpha^{-1}\gamma$.
\begin{equation*}
\xymatrix @R=0.75pc @C=0.75pc{
      &   & 2 \ar[ld]_\alpha \ar[rd]^\gamma &                 &   &                  & 2 \ar[ld]_\alpha \ar[rd]^\gamma &   \\
w_2 = & 1 &                                 & 1 \ar[rd]_\beta &   & 1 \ar[ld]^\delta &                                 & 1 \\
      &   &                                 &                 & 3 &                  &                                 &   \\
}
\qquad
\xymatrix @R=0.75pc @C=0.75pc{
      &   & 2 \ar[ld]_\alpha \ar[rd]^\gamma &   \\
w_3 = & 1 &                                 & 1 \\
      &   &                                 &   \\
} 
\end{equation*}
The corresponding string modules form a mutable triple $T_0 = (M(w_1),M(w_2),M(w_3))$. In this case, an element of $\mathcal{T}(T_0)$ will be called a \emph{Markoff module triple} and a module belonging to such a triple will be called a \emph{Markoff module}. Any Markoff module appearing as the middle term of a Markoff module triple is said to be \emph{proper}.

We know turn our attention to the behavior of Markoff module triples under the operations $\mu_R$ and $\mu_L$.
For a module $M$, let $\dimv M$ denote its dimension vector. Note that when $M(w)$ is a string module with $w=a_1a_2\cdots a_ n$, we have
$\dimv M(w) = \left( \delta_{t(a_n),j} + \sum_{i=1}^n \delta_{s(a_i),j}\right)_{j\in Q_0}$ and 
$\dimv M(\varepsilon_i) = \left( \delta_{i,j}\right)_{j\in Q_0}$ where $\delta_{i,j}$ is the Kronecker delta.

\begin{corollary}\label{mmdimv}
Let $T=(M_1,M_2,M_3)$ be a Markoff module triple.
\begin{enumerate}
  \item If $\mu_R(T) = (M_1,M'_3,M_2)$, then $\dimv M'_3 = 2 \dimv M_2 - \dimv M_3$.
  \item If $\mu_L(T) = (M_2,M'_1,M_3)$, then $\dimv M'_1 = 2 \dimv M_2 - \dimv M_1$.
\end{enumerate}
Moreover, if $M$ is a Markoff module with $\dimv M = (a,b,c)$ then $a-b-c=1$.
\end{corollary}
\begin{proof}
Properties (a) and (b) are a direct consequence of \cref{mtses}. As for the second claim, the condition holds for every module in the initial triple  $T_0$. Suppose that it is also true for the modules of some triple $T=(M_1,M_2,M_3)$. Consider $\mu_L(T) = (M_2,M'_1,M_3)$ and let $\dimv M'_1 = (a,b,c)$, $\dimv M_1 = (a_1,b_1,c_1)$ and $\dimv M_2 = (a_2,b_2,c_2)$. Then by (b) we have
\begin{align*}
  a-b-c &= (2a_2-a_1) - (2b_2-b_1) - (2c_2-c_1) \\
        &= 2 (a_2-b_2-c_2) - (a_1-b_1-c_1) \\
        &= 2\cdot 1 - 1 = 1
\end{align*}
The proof for $\mu_R(T)$ is similar.
\end{proof} 

\begin{proposition}\label{mmmtR2}
If $T=(M(w_1),M(w_2),M(w_3))$ is a Markoff module triple, then
\begin{enumerate}
  \item $w_2 = w_3u_1 = u_2w_3$ for some strings $u_1, u_2$ and where both occurrences of $w_3$ are as factor strings of $w_2$.
  \item $w_2 = w_1v_1 = v_2w_1$ for some strings $v_1, v_2$ and where both occurrences of $w_1$ are as substrings of $w_2$.
\end{enumerate}
Moreover, \[\mu_R(T) = (M(w_1),M(w'_3),M(w_2)) \text{ where } w'_3 = w_2u_1 = u_2w_2\] and
          \[\mu_L(T) = (M(w_2),M(w'_1),M(w_3)) \text{ where } w'_1 = w_2v_1 = v_2w_2\]
          where $u_1,u_2,v_1$ and $v_2$ are as above.
\end{proposition}
\begin{proof}The conditions can easily be checked for the initial triple $T_0$. Now suppose the properties hold for some $T=(M(w_1),M(w_2),M(w_3))$. Consider the maps $\alpha_1,\alpha_2:M(w_2)\to M(w_3)$ where $\alpha_1:w_2=w_3u_1\twoheadrightarrow w_3$ and 
$\alpha_2:w_2=u_2w_3\twoheadrightarrow w_3$. Clearly $g=\matmaph{\alpha_1}{\alpha_2}:M(w_2)\oplus M(w_2)\to M(w_3)$ is an epimorphism. Let $w'_3 = w_2u_1$, then we also have $w'_3 = u_2w_2$ since $w_2u_1 = u_2w_3u_1 = u_2w_2$. We define $\alpha'_1,\alpha'_2:M(w'_3)\to M(w_2)$ by $\alpha'_1:w'_3=w_2u_1\twoheadrightarrow w_2$ and $\alpha'_2:w'_3=u_2w_2\twoheadrightarrow w_2$. We have $\alpha_1 \alpha'_1 = \alpha_2 \alpha'_2$ and $f=\matmapv{\alpha'_1}{-\alpha'_2}:M(w'_3)\to M(w_2)\oplus M(w_2)$ is a monomorphism. Consequently, we obtain an exact sequence
\[ \SES{M(w'_3)}{f}{M(w_2)\oplus M(w_2)}{g}{M(w_3)} \]
By construction this is the exact sequence of \cref{mtses} and so $\mu_R(T) = (M(w_1),M(w'_3),M(w_2))$. 
Furthermore, we note that $w'_3 = w_2u_1 = w_1(v_1u_1)$ and $w'_3 = u_2w_2 = (u_2v_2)w_1$.
The proof for $\mu_L$ is similar.
\end{proof} 

\section{Christoffel Words and Markoff modules}

To every Markoff module $M$ we associate a pair of integers $\delta(M)$ in the following way: If $M$ has dimension vector $(a,b,c)$ then $M\mapsto \delta(M) = (a-2b+c, b-c)$.

\begin{lemma}\label{matgcd}
Let $a,b,c,d \in\mathbb{Z}$. If $\det \left[\begin{smallmatrix} a&b\\ c&d \end{smallmatrix}\right]=1$, then $\gcd(a+c,b+d) = 1$.
\end{lemma}
\begin{proof}
We have
\[ \det \left[\begin{smallmatrix} a&b\\ c&d \end{smallmatrix}\right] = ad-bc = (ad+cd)-(bc+cd) = d(a+c) - c(b+d) \]
The claim follows from B\'{e}zout's lemma.
\end{proof}

\begin{lemma}\label{cwmm}
Let $(M_1, M_2, M_3)$ be a Markoff module triple and let $\delta(M_i) = (x_i,y_i)$. Then
\begin{enumerate}[ref=\cref{cwmm}(\alph{*})]
  \item $\delta(M_2) = \delta(M_1) + \delta(M_3)$.
  \item $\det \left[\begin{smallmatrix} x_1&y_1\\ x_3&y_3 \end{smallmatrix}\right] = 1$.
  \item $\gcd(x_i,y_i) = 1$ for each $i=1,2,3$. \label{cwmmrp}
\end{enumerate}
\end{lemma}
\begin{proof}
For the initial triple we have $\delta(M_1) = (1,0)$, $\delta(M_2) = (1,1)$ and $\delta(M_3) = (0,1)$. The three conditions clearly hold.
Suppose that a triple $T = (M_1, M_2, M_3)$ satisfies all the conditions. We will show the same is true for $\mu_L(T) = (M_2, M'_1, M_3)$; the proof for $\mu_R(T)$ is similar. 

Let $\delta(M'_1)=(x,y)$, $\dimv M_i = (a_i, b_i, c_i)$ and $\dimv M'_1 = (a, b, c)$. By \cref{mmdimv}, we have $\dimv M'_1 = 2\dimv M_2 - \dimv M_1$ and so 
\begin{align*}
  x = a - 2b + c &= (2a_2-a_1) - 2(2b_2-b_1) + (2c_2-c_1) \\
                 &= 2 (a_2-2b_2+c_2) - (a_1-2b_1+c_1) \\
                 &= 2 x_2 - x_1
\end{align*}
and
\begin{align*}
  y = b - c &= (2b_2-b_1) - (2c_2-c_1) \\
            &= 2 (b_2-c_2) - (b_1-c_1) \\
            &= 2 y_2 - y_1
\end{align*}
By the inductive hypothesis, $x_2 = x_1 + x_3$ and $y_2 = y_1 + y_3$. Hence $2x_2-x_1 = x_2+x_3$ and $2y_2-y_1=y_2+y_3$. We get
\[ \delta(M'_1)=(2 x_2 - x_1, 2 y_2 - y_1) = (x_2+x_3,y_2+y_3) = \delta(M_2)+\delta(M_3) \]
Moreover,
\[ \left|\begin{array}{rr} x_2&y_2\\ x_3&y_3 \end{array}\right| = x_2y_3-y_2x_3 = (x_1 + x_3)y_3 - (y_1 + y_3)x_3 = \left|\begin{array}{rr} x_1&y_1\\ x_3&y_3 \end{array}\right| = 1 \]
It follows from \cref{matgcd} that $\gcd(x,y)=1$.
\end{proof}

By \ref{cwmmrp}, we can associate to each Markoff module $M$ the Christoffel word $C(\delta(M))$.
\begin{proposition}\label{mtcwsf}
If $(M_1, M_2, M_3)$ is a Markoff module triple then $C(\delta(M_2)) = C(\delta(M_1))C(\delta(M_3))$ is the standard factorisation of the Christoffel word $C(\delta(M_2))$.
\end{proposition}
\begin{proof}
A consequence of \cref{cwmm,cwcon}.
\end{proof}

Let $\mathcal{T}$ be the set of Markoff module triples and $\mathcal{C}$ the set of Christoffel triples. 
We define a map $F:\mathcal{T}\to\mathcal{C}$ by \[(M_1, M_2, M_3) \mapsto (C(\delta(M_1)), C(\delta(M_2)), C(\delta(M_3)))\]
\begin{theorem}The following diagrams commute
\[\xymatrix{
\mathcal{T} \ar[r]^{F} \ar[d]_{\mu_{L}} & \mathcal{C} \ar[d]^{c_L} & & \mathcal{T} \ar[r]^{F} \ar[d]_{\mu_{R}} & \mathcal{C} \ar[d]^{c_R} \\
\mathcal{T} \ar[r]^{F}                  & \mathcal{C}              & & \mathcal{T} \ar[r]^{F}                  & \mathcal{C}              \\
}\]
Moreover, $F$ is a binary tree isomorphism.
\end{theorem}
\begin{proof}
Let $T=(M_1, M_2, M_3)\in\mathcal{T}$, $\mu_L(T) = (M_2,M'_1,M_3)$ and $\mu_R(T) = (M_1,M'_3,M_2)$. 
The commutativity follows from the fact that $C(\delta(M'_1)) = C(\delta(M_2))C(\delta(M_3))$ and $C(\delta(M'_3)) = C(\delta(M_1))C(\delta(M_2))$ by \cref{mtcwsf}.

To show that $F$ is an isomorphism, it suffices, by \cref{btiso}, to show that $F(\text{root } \mathcal{T}) = \text{root }\mathcal{C}$. We have
\[F(\text{root } \mathcal{T}) = (C(1,0), C(1,1), C(0,1)) = (x,xy,y) = \text{root }\mathcal{C}\]
\end{proof}

\begin{corollary}\label{mmtun}
Every Markoff module triple is uniquely determined by its middle term.
\end{corollary}
\begin{proof}
Using the bijection $F$, this results from uniqueness of the standard factorisation of Christoffel words.
\end{proof}

\section{Markoff Modules and Markoff Triples}

Consider the monoid homomorphism
$\rho : \{1,2,3\}^* \to \SL_2(\mathbb{Z})$ defined by
\[ \rho(1) = \left[\begin{array}{rr} 2&1\\ 1&1 \end{array}\right]   \quad
   \rho(2) = \left[\begin{array}{rr} 2&-1\\ -1&1 \end{array}\right] \quad
   \rho(3) = \left[\begin{array}{rr} 0&-1\\ 1&3 \end{array}\right] \]
Similar matrices (with reversed diagonals) appear in Cohn's study of Markoff forms (see \cite{C1}, \cite{C2}).

Let $\mathcal{S}$ be the set of strings for some bound quiver $(Q,I)$. Define a map $\nu:\mathcal{S}\to Q_0^*$ by
\[ \nu(w)= \left\{ \begin{array}{ll} s(a_1)s(a_2)\cdots s(a_n)t(a_n) & \mbox{if $w=a_1a_2\cdots a_n$} \\
                                     i                               & \mbox{if $w=\varepsilon_i$}
                   \end{array} \right. \]

When $Q_0 = \set{1,2,3}$, we define $\phi = \rho\nu$. Note that when the concatenation $vw$ of $v$ and $w$ is defined then $\phi(vw) = \phi(v)\phi(\varepsilon_i)^{-1}\phi(w)$ where $i$ is the end point of $v$ (and starting point of $w$).

\begin{lemma}[Fricke identities \cite{Fri}]
For every $A,B \in \SL_2(\mathbb{Z})$,
\begin{enumerate}
  \item $\Tr(A)^2 + \Tr(B)^2 + \Tr(AB)^2 = \Tr(A)\Tr(B)\Tr(AB) + \Tr(ABA^{-1}B^{-1}) + 2$
  \item $\Tr(AB^2) + \Tr(A) = \Tr(AB)\Tr(B)$
\end{enumerate}
\end{lemma}

\begin{proposition}
Let $(M(w_1),M(w_2),M(w_3))$ be a Markoff module triple.
\begin{enumerate}
  \item If $\mu_L(T) = (M(w_2),M(w'_1),M(w_3))$, then \[\phi(w'_1) = \phi(w_2)\phi(w_1)^{-1}\phi(w_2)\]
  \item If $\mu_R(T) = (M(w_1),M(w'_3),M(w_2))$, then \[\phi(w'_3) = \phi(w_2)\phi(w_3)^{-1}\phi(w_2)\]
\end{enumerate}
\end{proposition}
\begin{proof}By \cref{mmmtR2}, $w_2=w_1v$ and $w'_1 = w_2v$ for some string $v$. From the first equality we get
\[ \phi(w_2) = \phi(w_1)\phi(\varepsilon_i)^{-1}\phi(v) \]
where $i$ is the end point of $w_1$. Hence
\[ \phi(w_1)^{-1}\phi(w_2) = \phi(\varepsilon_i)^{-1}\phi(v) \]
Now by using the second equality,
\[ \phi(w'_1) = \phi(w_2)\phi(\varepsilon_i)^{-1}\phi(v) = \phi(w_2)\phi(w_1)^{-1}\phi(w_2)\]
The proof of the second statement is similar.
\end{proof}

\begin{corollary}\label{mmmt12}
Let $(M(w_1),M(w_2),M(w_3))$ be a Markoff module triple.
The matrices $\phi(w_i)$ satisfy $\tfrac{1}{3}\Tr\phi(w_i) = \phi(w_i)_{12}$.
\end{corollary}
\begin{proof} One can easily verify that for every $A,B \in \SL_2(\mathbb{Z})$ if $\tfrac{1}{3}\Tr(A) = A_{12}$ and $\tfrac{1}{3}\Tr(B) = B_{12}$
then $\tfrac{1}{3}\Tr(AB^{-1}A) = (AB^{-1}A)_{12}$.
\end{proof}

\begin{proposition}\label{mmmtMult}
If $(M(w_1),M(w_2),M(w_3))$ is a Markoff module triple, then \[\phi(w_2)=\phi(w_1)\phi(w_3)\]
\end{proposition}
\begin{proof}
For the initial triple, the formula can be verified by a simple calculation. Now suppose the proposition holds for some triple $T = (M(w_1),M(w_2),M(w_3))$, so that $\phi(w_2)=\phi(w_1)\phi(w_3)$. Let $\mu_L(T) = (M(w_2),M(w'_1),M(w_3))$, we then have \[ \phi(w'_1) = \phi(w_2)\phi(w_1)^{-1}\phi(w_2) = \phi(w_2)\phi(w_1)^{-1}\phi(w_1)\phi(w_3) = \phi(w_2)\phi(w_3)\]
The same argument can be used for $\mu_R(T)$.
\end{proof}

\begin{corollary}\label{mmmtPos}
Let $(M(w_1),M(w_2),M(w_3))$ be a Markoff module triple. The matrices $\phi(w_i)$ have strictly positive entries.
\end{corollary}
\begin{proof}
True for the initial triple. Proceed by induction using \cref{mmmtMult} and 
noting that if $(M(w_1),M(w_2),M(w_3))$ is non-initial, then $M(w_1)$ and $M(w_3)$ appear in the antecedent triple.
\end{proof}

\begin{proposition}\label{mmmtTri}
If $(M(w_1),M(w_2),M(w_3))$ is a Markoff module triple, then
\[ \left( \tfrac{1}{3}\Tr\phi(w_1),\, \tfrac{1}{3}\Tr\phi(w_2),\, \tfrac{1}{3}\Tr\phi(w_3) \right) \]
is a proper Markoff triple.
\end{proposition}
\begin{proof}
Explicit calculation shows that the initial triple yields the proper Markoff triple $(1,5,2)$. Suppose that the statement is true for some
 $T = (M(w_1),M(w_2),M(w_3))$. Let $A=\phi(w_1)$, $B=\phi(w_3)$ and $C=\phi(w_2)$. By the Fricke identity we have
\[ \Tr(A)^2 + \Tr(B)^2 + \Tr(AB)^2 = \Tr(A)\Tr(B)\Tr(AB) + \Tr(ABA^{-1}B^{-1})+2\]
and since $C=AB$ by \cref{mmmtMult},
\[ \Tr(A)^2 + \Tr(B)^2 + \Tr(C)^2 = \Tr(A)\Tr(B)\Tr(C) + \Tr(ABA^{-1}B^{-1})+2 \]
Then, our assumption that $\left( \tfrac{1}{3}\Tr(A),\, \tfrac{1}{3}\Tr(B),\, \tfrac{1}{3}\Tr(C) \right)$ is a Markoff triple implies
\[ \Tr(A)^2 + \Tr(B)^2 + \Tr(C)^2 = \Tr(A)\Tr(B)\Tr(C) \]
Hence $\Tr(ABA^{-1}B^{-1}) = -2$.

Let $(M(w_2),M(w'_1),M(w_3)) = \mu_L(T)$ and $A' = \phi(w'_1)$. We then have
\[ \Tr(C)^2 + \Tr(B)^2 + \Tr(CB)^2 = \Tr(C)\Tr(B)\Tr(CB)+\Tr(CBC^{-1}B^{-1})+2 \]
but since $A'=CB$,
\[ \Tr(C)^2 + \Tr(B)^2 + \Tr(A')^2 = \Tr(C)\Tr(B)\Tr(A')+\Tr(CBC^{-1}B^{-1})+2 \]
Moreover, from $C=AB$, we get $B=A^{-1}C$ and so
\[ CBC^{-1}B^{-1} = CA^{-1}CC^{-1}B^{-1} = ABA^{-1}B^{-1} \]
Thus $\Tr(CBC^{-1}B^{-1}) = \Tr(ABA^{-1}B^{-1}) = -2$. We conclude that
\[\left( \tfrac{1}{3}\Tr(C),\, \tfrac{1}{3}\Tr(A'),\, \tfrac{1}{3}\Tr(B) \right)\]
is a Markoff triple. We still need to show that it is proper. The inductive hypothesis implies that $\tfrac{1}{3}\Tr(B) \neq \tfrac{1}{3}\Tr(C)$. Using \cref{mmmt12,mmmtPos}, we get
\begin{align*}
  \tfrac{1}{3}\Tr(A') = A'_{12} &= C_{11}B_{12}+C_{12}B_{22}  \\
                                &> B_{12}+C_{12}              \\
                                &= \tfrac{1}{3}\Tr(B) + \tfrac{1}{3}\Tr(C)
\end{align*}
from which we deduce
$\tfrac{1}{3}\Tr(A') > \tfrac{1}{3}\Tr(B)$ and $\tfrac{1}{3}\Tr(A') > \tfrac{1}{3}\Tr(C)$.
\end{proof}

Let $\mathcal{T}$ be the set of Markoff module triples. From \cref{mmmtTri}, we get a map $\Phi:\mathcal{T}\to\mathcal{M}$ defined on a triple $T = (M(w_1),M(w_2),M(w_3))$ by
\[ \Phi(T) = \left( \tfrac{1}{3}\Tr\phi(w_1),\, \tfrac{1}{3}\Tr\phi(w_2),\, \tfrac{1}{3}\Tr\phi(w_3) \right) \]
\begin{theorem}
The following diagrams commute
\[ \xymatrix{
\mathcal{T} \ar[r]^{\Phi} \ar[d]_{\mu_{L}} & \mathcal{M} \ar[d]^{m_L}  & & \mathcal{T} \ar[r]^{\Phi} \ar[d]_{\mu_{R}} & \mathcal{M} \ar[d]^{m_R} \\
\mathcal{T} \ar[r]^{\Phi}                   & \mathcal{M}              & & \mathcal{T} \ar[r]^{\Phi}                  & \mathcal{M}              \\
} \]
Moreover, $\Phi$ is a binary tree isomorphism.
\end{theorem}
\begin{proof}
Let $T = (M(w_1),M(w_2),M(w_3))\in\mathcal{T}$. Then $\Phi(T) = (a,b,c)$ where
\[ a=\tfrac{1}{3}\Tr\phi(w_1),\quad b=\tfrac{1}{3}\Tr\phi(w_2) \quad\text{and}\quad c=\tfrac{1}{3}\Tr\phi(w_3)\]
Since $m_L\Phi(T) = (b,3bc-a,c)$ and
\[\mu_{L}(T) = (M(w_2),M(w'_1),M(w_3))\] it suffices to prove that $\tfrac{1}{3}\Tr(\phi(w'_1)) = 3bc-a$.
By \cref{mmmtMult}, $\phi(w'_1) = \phi(w_2)\phi(w_3)$ and $\phi(w_2) = \phi(w_1)\phi(w_3)$ so
\[\phi(w'_1) = \phi(w_1)\phi(w_3)\phi(w_3)\]
thus, by the second Fricke identity,
\begin{align*}
  \Tr\phi(w'_1) &= \Tr(\phi(w_1)\phi(w_3))\Tr\phi(w_3) - \Tr\phi(w_1) \\
                &= (\Tr\phi(w_2))(\Tr\phi(w_3)) - \Tr\phi(w_1) \\
                &= 3c \cdot 3b - 3a \\
                &= 3 (3bc-a)
\end{align*}
The second case is proved similarly.

That $\Phi$ is an isomorphism is a consequence of \cref{btiso} since
$\Phi(\text{root } \mathcal{T}) = (1,5,2) = \text{root }\mathcal{M}$.
\end{proof}

\begin{corollary}
The uniqueness conjecture for Markoff numbers is equivalent to the injectivity of the map $M(w)\mapsto \tfrac{1}{3}\Tr\phi(w)$ where $M(w)$ is a  proper Markoff module.
\end{corollary}
\begin{proof}Suppose that the uniqueness conjecture holds. Let $M(w_1)$ and $M(w_2)$ be two proper Markoff modules such that $\Tr\phi(w_1) = \Tr\phi(w_2)$. Since $M(w_1)$ and $M(w_2)$ are proper, there exist $T_1,T_2\in\mathcal{T}$ such that $M(w_1)$ and $M(w_2)$ are the middle terms of $T_1$ and $T_2$ respectively. Let $\Phi(T_1)=(a_1,b_1,c_1)$ and $\Phi(T_2)=(a_2,b_2,c_2)$.
We have $b_1 = \tfrac{1}{3}\Tr\phi(w_1) = \tfrac{1}{3}\Tr\phi(w_2) = b_2$ and we supposed that every Markoff triple is uniquely determined by its largest term, thus $\Phi(T_1)=\Phi(T_2)$ and consequently $M(w_1)=M(w_2)$.

Conversely, it is known that the Markoff triples $(1,1,1)$ and $(1,2,1)$ are uniquely determined by their largest terms. Let $m_1=(a_1,b_1,c_1)$ and $m_2=(a_2,b_2,c_2)$ be two proper Markoff triples such that $b_1=b_2$. Since $\Phi$ is bijective, there are Markoff module triples $T_1 = (M(w_1),M(w_2),M(w_3))$ and $T_2 = (M(v_1),M(v_2),M(v_3))$ such that $\Phi(T_1)=m_1$ and $\Phi(T_2)=m_2$. In particular, $\tfrac{1}{3}\Tr\phi(w_2) = \tfrac{1}{3}\Tr\phi(v_2)$ and then by our assumption $M(w_2) = M(v_2)$. By \cref{mmtun}, $T_1=T_2$ hence $m_1=m_2$.
\end{proof}


\end{document}